\documentclass[11pt,leqno]{article}
\usepackage{hyperref}
\usepackage{soul}
\usepackage{pdfsync}
\usepackage[doc]{optional}
\usepackage{xcolor}
\definecolor{labelkey}{rgb}{0,0.08,0.45}
\definecolor{rekey}{rgb}{0,0.6,0.0}
\definecolor{Brown}{rgb}{0.45,0.0,0.05}
\usepackage{exscale,relsize}
\usepackage{amsmath}
\usepackage{amsfonts}
\usepackage{amssymb}
\usepackage{calc}
\usepackage{theorem}
\usepackage{pifont}      
\usepackage{graphicx}
\DeclareMathOperator{\weakstarly}{\rightharpoondown_{\mathrm{w*}}}

\newcommand{\scal}[2]{\langle{{#1},{#2}}\rangle}

\newcommand{\RR}{\ensuremath{\mathbb R}}

\newcommand{\RX}{\ensuremath{\,\left]-\infty,+\infty\right]}}

\newcommand{\NN}{\ensuremath{\mathbb N}}

\newcommand{\menge}[2]{\big\{{#1} \mid {#2}\big\}}

\newcommand{\To}{\ensuremath{\rightrightarrows}}

\newcommand{\aff}{\operatorname{aff}}

\newcommand{\dom}{\ensuremath{\operatorname{dom}}}

\newcommand{\gra}{\ensuremath{\operatorname{gra}}}

\newcommand{\intdom}{\ensuremath{\operatorname{int}\operatorname{dom}}\,}
\newcommand{\inte}{\ensuremath{\operatorname{int}}}

\newcommand{\bd}{\ensuremath{\operatorname{bdry}}}

\renewcommand{\phi}{\ensuremath{\varphi}}

\newtheorem{theorem}{Theorem}[section]
\newtheorem{lemma}[theorem]{Lemma}
\newtheorem{fact}[theorem]{Fact}
\newtheorem{corollary}[theorem]{Corollary}
\newtheorem{proposition}[theorem]{Proposition}

\theoremstyle{plain}{\theorembodyfont{\rmfamily}
}
\theoremstyle{plain}{\theorembodyfont{\rmfamily}
}
\theoremstyle{plain}{\theorembodyfont{\rmfamily}
}
\theoremstyle{plain}{\theorembodyfont{\rmfamily}
}
\theoremstyle{plain}{\theorembodyfont{\rmfamily}
\newtheorem{remark}[theorem]{Remark}}
\newtheorem{problem}[theorem]{Open problem}
\theoremstyle{plain}{\theorembodyfont{\rmfamily}
}

\begin{document}


\title{\sffamily{Maximality of the sum of a maximally monotone
 linear relation and a maximally monotone operator}\\
{\small Dedicated to Petar Kenderov on the occasion of his seventieth birthday}}

\author{
Jonathan M. Borwein\thanks{CARMA, University of Newcastle,
Newcastle, New South Wales 2308, Australia. E-mail:
\texttt{jonathan.borwein@newcastle.edu.au}. Laureate Professor at
the University of Newcastle and Distinguished Professor at  King
Abdul-Aziz University, Jeddah.}\; and Liangjin\ Yao\thanks{CARMA,
University of Newcastle,
 Newcastle, New South Wales 2308, Australia.
E-mail:  \texttt{liangjin.yao@newcastle.edu.au}.}}

\date{December 18,  2012}
\maketitle

\begin{abstract} \noindent
The most famous open problem in Monotone Operator Theory
concerns the maximal monotonicity of the sum of two
maximally monotone operators provided that
Rockafellar's constraint qualification holds.

In this paper, we prove the maximal monotonicity of $A+B$ provided
that $A, B$ are maximally monotone and $A$ is a linear relation, as
soon as
 Rockafellar's constraint qualification holds:
$\dom A\cap\inte\dom B\neq\varnothing$. Moreover, $A+B$ is of type (FPV).

\end{abstract}



\section{Introduction}

Throughout this paper, we assume that
$X$ is a real Banach space with norm $\|\cdot\|$,
that $X^*$ is the continuous dual of $X$, and
that $X$ and $X^*$ are paired by $\scal{\cdot}{\cdot}$.
Let $A\colon X\To X^*$
be a \emph{set-valued operator}  (also known as a relation, point-to-set mapping or multifunction)
from $X$ to $X^*$, i.e., for every $x\in X$, $Ax\subseteq X^*$,
and let
$\gra A := \menge{(x,x^*)\in X\times X^*}{x^*\in Ax}$ be
the \emph{graph} of $A$.
Recall that $A$ is  \emph{monotone} if
\begin{equation}
\scal{x-y}{x^*-y^*}\geq 0,\quad \forall (x,x^*)\in \gra A\;\forall (y,y^*)\in\gra A,
\end{equation}
and \emph{maximally monotone} if $A$ is monotone and $A$ has no proper monotone extension
(in the sense of graph inclusion).
Let $A:X\rightrightarrows X^*$ be monotone and $(x,x^*)\in X\times X^*$.
 We say $(x,x^*)$ is \emph{monotonically related to}
$\gra A$ if
\begin{align*}
\langle x-y,x^*-y^*\rangle\geq0,\quad \forall (y,y^*)\in\gra A.\end{align*}
Let $A:X\rightrightarrows X^*$ be maximally monotone. We say $A$ is
\emph{of type (FPV)} if  for every open convex set $U\subseteq X$ such that
$U\cap \dom A\neq\varnothing$, the implication
\begin{equation*}
x\in U\,\text{and}\,(x,x^*)\,\text{is monotonically related to $\gra A\cap (U\times X^*)$}
\Rightarrow (x,x^*)\in\gra A
\end{equation*}
holds. We say $A$ is a \emph{linear relation} if $\gra A$ is a
linear subspace. Monotone operators have proven to be important
objects in modern Optimization and Analysis; see, e.g., the books
\cite{BC2011,BorVan,BurIus,ButIus,ph,Si,Si2,RockWets,Zalinescu,Zeidler2A,Zeidler2B}
and the references therein. We adopt standard notation used in these
books: $\dom A:= \menge{x\in X}{Ax\neq\varnothing}$ is the
\emph{domain} of $A$. Given a subset $C$ of $X$, $\inte C$ is the
\emph{interior} of $C$, $\bd{C}$ is the \emph{boundary} of $C$,
$\aff C$ is the \emph{affine hull} of $C$ and $\overline{C}$ is the
norm \emph{closure} of $C$. We set $C^{\bot}:= \{x^*\in X^*
\mid(\forall c\in C)\, \langle x^*, c\rangle=0\}$ and $S^{\bot}:=
\{x^{**}\in X^{**} \mid(\forall s\in S)\, \langle x^{**},
s\rangle=0\}$ for a set  $S\subseteq X^*$. We define $^{ic}C$ by
\begin{equation*}
^{ic}C:=\begin{cases}^{i}C,\,&\text{if $\aff C$ is closed};\\
\varnothing,\,&\text{otherwise},
\end{cases}\end{equation*} where $^{i}C$ \cite{Zalinescu}
 is the \emph{intrinsic core} or \emph{relative algebraic interior} of $C$,  defined by
$^{i}C:=\{a\in C\mid \forall x\in \aff(C-C),
\exists\delta>0, \forall\lambda\in\left[0,\delta\right]:
a+\lambda x\in C\}$.

The \emph{indicator function} of $C$, written as $\iota_C$, is defined
at $x\in X$ by
\begin{align}
\iota_C (x):=\begin{cases}0,\,&\text{if $x\in C$;}\\
\infty,\,&\text{otherwise}.\end{cases}\end{align}
If $D\subseteq X$, we set $C-D=\{x-y\mid x\in C, y\in D\}$.
  For every $x\in X$, the normal cone operator of $C$ at $x$
is defined by $N_C(x):= \menge{x^*\in
X^*}{\sup_{c\in C}\scal{c-x}{x^*}\leq 0}$, if $x\in C$; and $N_C(x)=\varnothing$,
if $x\notin C$.
For $x,y\in X$, we set $\left[x,y\right]:=\{tx+(1-t)y\mid 0\leq t\leq 1\}$.
 Given $f\colon X\to \RX$, we set
$\dom f:= f^{-1}(\RR)$
. We say $f$ is \emph{proper} if $\dom f\neq\varnothing$.
Let $f$ be proper.
Then
   $\partial f\colon X\To X^*\colon
   x\mapsto \menge{x^*\in X^*}{(\forall y\in
X)\; \scal{y-x}{x^*} + f(x)\leq f(y)}$ is the \emph{subdifferential
operator} of $f$. We also set $P_X: X\times X^*\rightarrow X\colon
(x,x^*)\mapsto x$. Finally,  the \emph{open unit ball} in $X$ is
denoted by $U_X:= \menge{x\in X}{\|x\|< 1}$, the \emph{closed unit
ball} in $X$ is denoted by $B_X:= \menge{x\in X}{\|x\|\leq 1}$, and
$\NN:=\{1,2,3,\ldots\}$. We denote by $\longrightarrow$ and
$\weakstarly$ the norm convergence and weak$^*$ convergence of
nets,  respectively.

Let $A$ and $B$ be maximally monotone operators from $X$ to
$X^*$.
Clearly, the \emph{sum operator} $A+B\colon X\To X^*\colon x\mapsto
Ax+Bx: = \menge{a^*+b^*}{a^*\in Ax\;\text{and}\;b^*\in Bx}$
is monotone.
Rockafellar established the following very important result in 1970.
\begin{theorem}[Rockafellar's sum theorem]
\emph{(See  \cite[Theorem~1]{Rock70} or \cite{BorVan}.)} Suppose
that $X$ is reflexive. Let $A, B: X\rightrightarrows  X^*$ be
maximally monotone. Assume that $A$ and  $B$  satisfy the classical
\emph{constraint qualification}\[\dom A \cap\intdom B\neq
\varnothing.\] Then $A+B$ is maximally monotone.
\end{theorem}

The most significant open problem in the theory concerns  the
maximal monotonicity of the sum of two maximally monotone operators
in general Banach spaces, which is called the ``sum problem''. Some
recent developments on the sum problem can be found in  Simons'
monograph \cite{Si2} and \cite{Bor1,Bor2,Bor3,BorVan, BY2, ZalVoi,
MarSva5,VV2,Yao3,Yao2,YaoPhD}. It is known, among other things, that
the sum theorem holds under Rockafellar's constraint qualification
when both operators are of dense type or when each operator has
nonempty domain interior \cite[Ch. 8]{BorVan} and \cite{ZalVoi2}.

Here we focus on the  case when $A$ is a maximally monotone linear
relation, and $B$ is maximally monotone
 such that
$\dom A\cap\inte\dom B\neq\varnothing$. In Theorem \ref{PGV:1} we
shall  show that $A+B$ is maximally monotone.

The remainder of this paper is organized as follows. In
Section~\ref{s:aux}, we collect auxiliary results for future
reference and for the reader's convenience. The proof of our main
result (Theorem~\ref{PGV:1}) forms the bulk of Section~\ref{s:main}.
In Section \ref{sec:more} various other consequences and related results are presented.

\section{Auxiliary Results}
\label{s:aux}

We start with a well known result from Rockafellar.
\begin{fact}[Rockafellar]
\emph{(See \cite[Theorem~1]{Rock69} or \cite[Theorem~27.1 and Theorem~27.3]{Si2}.)}
\label{f:referee02c}
Let $A:X\To X^*$ be  maximal monotone with $\inte\dom A\neq\varnothing$. Then
$\inte\dom A=\inte\overline{\dom A}$ and $\overline{\dom A}$ is convex.
\end{fact}

The Fitzpatrick function below is a very useful tool in Monotone
Operator  Theory, which by now has been applied comprehensively.
\begin{fact}[Fitzpatrick]
\emph{(See {\cite[Corollary~3.9]{Fitz88}}.)}
\label{f:Fitz}
Let $A\colon X\To X^*$ be maximally monotone,  and set
\begin{equation}
F_A\colon X\times X^*\to\RX\colon
(x,x^*)\mapsto \sup_{(a,a^*)\in\gra A}
\big(\scal{x}{a^*}+\scal{a}{x^*}-\scal{a}{a^*}\big),
\end{equation}
 the \emph{Fitzpatrick function} associated with $A$.
Then for every $(x,x^*)\in X\times X^*$, the inequality
$\scal{x}{x^*}\leq F_A(x,x^*)$ is true, and  equality holds if and
only if $(x,x^*)\in\gra A$.
\end{fact}

\begin{fact}
\emph{(See \cite[Theorem~3.4 and Corollary~5.6]{Voi1}, or \cite[Theorem~24.1(b)]{Si2}.)}
\label{f:referee1}
Let $A, B:X\To X^*$ be maximally monotone operators. Assume
$\bigcup_{\lambda>0} \lambda\left[P_X(\dom F_A)-P_X(\dom F_B)\right]$
is a closed subspace.
If
\begin{equation}
F_{A+B}\geq\langle \cdot,\,\cdot\rangle\;\text{on \; $X\times X^*$},
\end{equation}
then $A+B$ is maximally monotone.
\end{fact}

We next introduce some properties of type (FPV) operators.

\begin{fact}[Simons]
\emph{(See \cite[Theorem~46.1]{Si2}.)}
\label{S:referee01}
Let $A:X\To X^*$ be a maximally monotone linear relation.
Then $A$ is of type (FPV).
\end{fact}

\begin{fact}[Simons and Verona-Verona]
\emph{(See \cite[Theorem~44.1]{Si2} or \cite{VV1}.)}
\label{f:referee02a} Let $A:X\To X^*$ be maximally monotone. Suppose
that for every closed convex subset $C$ of $X$ with $\dom A \cap
\inte C\neq \varnothing$, the operator $A+N_C$ is maximally
monotone. Then $A$ is of type  (FPV).
\end{fact}

Next we present a useful sufficient condition for the sum problem to
have a positive resolution (see also \cite{Bor2}).

\begin{fact}[Voisei and Z\u{a}linescu]
\emph{(See  \cite[Corollary~4]{ZalVoi}.)}
\label{f:referee01}
Let $A, B:X\To X^*$ be  maximally monotone.  Assume that $^{ic}(\dom A)\neq\varnothing,
^{ic}(\dom B)\neq\varnothing$ and $0\in^{ic}\left[\dom A-\dom B\right]$.
Then $A+B$ is maximally monotone.
\end{fact}

\begin{fact}\emph{(See \cite[Lemma~2.9]{BWY9}.)}\label{rcf:001}
Let $A:X\To X^*$ be a maximally monotone linear relation,
and let $z\in X\cap (A0)^\bot$.
Then $z\in\overline{\dom A}$.
\end{fact}

\begin{fact}\label{Ll:l1}\emph{(See \cite[Lemma~2.5]{BWY4}.)}
Let $C$ be  a nonempty closed convex
subset of $X$ such that $\inte C\neq \varnothing$.
Let $c_0\in \inte C$ and suppose that $z\in X\smallsetminus C$.
Then there exists
$\lambda\in\left]0,1\right[$ such
that $\lambda c_0+(1-\lambda)z\in\bd C$.
\end{fact}

\begin{fact}[Boundedness below]\emph{(See  \cite[Fact~4.1]{BY1}.)}\label{extlem}
Let $A:X\rightrightarrows X^*$ be monotone and $x\in\inte\dom A$. Then there exist $\delta>0$ and  $M>0$ such that
 $x+\delta B_X\subseteq\dom A$ and $\sup_{a\in x+\delta B_X}\|Aa\|\leq M$.
Assume that $(z,z^*)$ is monotonically related to $\gra A$. Then
\begin{align}
\langle z-x, z^*\rangle
\geq \delta\|z^*\|-(\|z-x\|+\delta) M.
\end{align}
\end{fact}

Before we turn to our main result, we need the following technical
lemma.

\begin{lemma}\label{rcf:01}
Let $A:X\To X^*$ be a  monotone linear relation, and let
$B:X\rightrightarrows X^*$ be a maximally monotone operator. Suppose
that $\dom A \cap \inte \dom B\neq \varnothing$. Suppose also that
$(z,z^*)\in X\times X^*$ is monotonically related to $\gra (A+ B)$,
and that $z\in\dom A$. Then  $z\in  \dom B$.
\end{lemma}
\begin{proof}
We can and do suppose that $(0,0)\in\gra A\cap\gra B$ and $0\in\dom A \cap \inte \dom B$.
Suppose to the contrary that $z\notin\dom B$. Then we have $z\neq0$.
We claim that
\begin{align}
N_{\left[0,z\right]}+B \quad\text{is maximally monotone}.\label{Lrtg:1}
\end{align}
Since $z\neq0$, then we have $\tfrac{1}{2}z\in{^{ic}(\dom N_{\left[0,z\right]})}$. Clearly,
$^{ic}(\dom B)\neq\varnothing$ and $0\in{^{ic}\left[\dom A-\dom B\right]}$.
By Fact~\ref{f:referee01},
$N_{\left[0,z\right]}+B$ is maximally monotone and hence \eqref{Lrtg:1} holds.
Since $(z, z^*)\notin\gra (N_{\left[0,z\right]}+B)$, there exist $\lambda \in\left[0,1\right]$ and $x^*, y^*\in X^*$ such that $(\lambda z, x^*)\in\gra N_{\left[0,z\right]}$, $(\lambda z, y^*)\in\gra B$ and
\begin{align}
\langle z-\lambda z, z^*-x^*-y^*\rangle<0.\label{Lrtg:2}
\end{align}

Since $(\lambda z, x^*)\in\gra B$ and $z\notin\dom B$, $\lambda<1$. Then by
\eqref{Lrtg:2},
\begin{align}
\langle  z, -x^*\rangle+\langle  z, z^*-y^*\rangle=\langle  z, z^*-x^*-y^*\rangle<0.
\label{Lrtg:3}
\end{align}
Since $(\lambda z, x^*)\in\gra N_{\left[0,z\right]}$, we have $\langle z-\lambda z, x^*\rangle\leq0$. Then  $\langle z, -x^*\rangle\geq0$. Thus
\eqref{Lrtg:3} implies that
\begin{align}
\langle  z, z^*-y^*\rangle<0.
\label{Lrtg:4}
\end{align}
Let $a^*\in A(\lambda z)$. By the assumption, we have
\begin{align*}
\langle  z-\lambda z, z^*- a^*-y^*\rangle\geq0.
\end{align*}
Then we have $\langle z, z^*- a^*-y^*\rangle\geq0$ and hence
\begin{align}
\langle z, z^*-y^*\rangle\geq\langle z, a^*\rangle.\label{Lrtg:5}
\end{align}
Now we show that
\begin{align}
\langle z, a^*\rangle\geq0.\label{Lrtg:6}
\end{align}
We consider two cases.

\emph{Case 1}: $\lambda=0$. Then $a^*\in A0$. Since $ z\in\dom A$
and $A$ is monotone, \cite[Proposition~5.1(i)]{BBWY5} implies that
$\langle z, a^*\rangle=0$. Hence \eqref{Lrtg:5} holds.

\emph{Case 2}: $\lambda\neq0$. Since $(\lambda z, a^*)\in \gra A$,
$\langle \lambda z, a^*\rangle\geq0$ and hence $\langle  z,
a^*\rangle\geq0$. Hence \eqref{Lrtg:6} holds.

Combining \eqref{Lrtg:5} and \eqref{Lrtg:6},
\begin{align*}
\langle z, z^*-y^*\rangle\geq0,\quad\text{which contradicts \eqref{Lrtg:4}}.
\end{align*}
Hence $z\in\dom B$.
\end{proof}

\begin{remark}
Lemma~\ref{rcf:01} generalizes \cite[Lemma~2.10]{BWY9} in which $B$
is assumed to be a convex subdifferential.
\end{remark}

 We now come to our central result.

\section{Main Result}
\label{s:main}

The proof of Theorem~\ref{PGV:1} in part follows that of
\cite[Theorem~3.1]{Yao2}.

\begin{theorem}[Linear sum theorem]\label{PGV:1}
Let $A:X\To X^*$ be a maximally monotone linear relation, and let
$B: X\rightrightarrows X^*$ be maximally monotone. Suppose  that
$\dom A\cap\inte\dom B\neq\varnothing$.  Then $A+B$ is maximally
monotone.
\end{theorem}
\begin{proof} After translating the graphs if necessary, we can and do assume that
$0\in\dom A\cap\inte\dom B$ and that $(0,0)\in\gra A\cap\gra B$.
By Fact~\ref{f:Fitz}, $\dom A\subseteq P_X(\dom F_A)$ and
 $\dom B\subseteq P_X(\dom F_{B})$.
Hence,
\begin{align}\bigcup_{\lambda>0} \lambda
\big(P_X(\dom F_A)-P_X(\dom F_{B})\big)=X.\end{align}
Thus, by Fact~\ref{f:referee1}, it suffices to show that
\begin{equation} \label{e0:ourgoal}
F_{A+ B}(z,z^*)\geq \langle z,z^*\rangle,\quad \forall(z,z^*)\in X\times X^*.
\end{equation}
Take $(z,z^*)\in X\times X^*$.
Then
\begin{align}
&F_{A+B}(z,z^*)\nonumber\\
&=\sup_{\{x,x^*,y^*\}}\left[\langle x,z^*\rangle+\langle z,x^*\rangle-\langle x,x^*\rangle
+\langle z-x, y^*\rangle -\iota_{\gra A}(x,x^*)-\iota_{\gra
B}(x,y^*)\right].\label{see:1}
\end{align}
Assume to the contrary  that
\begin{align}
F_{A+B}(z,z^*)+\lambda<\langle z,z^*\rangle,\label{See:1a4}
\end{align}
where $\lambda>0$.

Now by \eqref{See:1a4},
\begin{align}
(z,z^*)\,\text{ is monotonically related to $\gra (A+B)$}.\label{SDF:47}\end{align}

We claim that \begin{align}z\notin \dom A.\label{LSD:1}\end{align}
 Indeed, if $z\in\dom A$, apply \eqref{SDF:47} and Lemma~\ref{rcf:01}
to get $z\in \dom B$. Thus $z\in\dom A\cap\dom B$ and hence
$F_{A+B}(z,z^*)\geq\langle z,z^*\rangle$ which contradicts
\eqref{See:1a4}. This establishes \eqref{LSD:1}.

By \eqref{See:1a4} and the assumption that $(0,0)\in\gra A\cap\gra B$,
we have
\begin{align*}
\sup\left[\langle 0,z ^*\rangle+\langle z,A0\rangle
-\langle 0, A0\rangle+\langle z, B0\rangle\right]
= \sup_{a^*\in A0,b^*\in B0}
\left[\langle z,a^*\rangle+\langle z, b^*\rangle\right]
<\langle z,z^*\rangle.
\end{align*}
Thus, since $A0$ is a linear subspace,
\begin{align}z\in X\cap (A0)^\bot.\label{ree:2}\end{align}
Then, by
Fact~\ref{rcf:001}, we have
\begin{align}z\in\overline{\dom A}.\label{ree:5}\end{align}
 Combining \eqref{LSD:1} and \eqref{ree:5}, we have
 \begin{align}
 z\in\overline{\dom A}\backslash{\dom A}.\label{LSD:2}\end{align}
Set
 \begin{align}
U_n:=  z+\tfrac{1}{n}U_X,\quad \forall n\in\NN\label{FD:1}.
\end{align}
By \eqref{LSD:2}, $(z,z^*)\notin\gra A$ and $U_n\cap\dom A\neq\varnothing$.
  Since $z\in U_n$ and $A$ is type of (FPV)
 by Fact~\ref{S:referee01},
there exists $(a_n, a^*_n)_{n\in\NN}$ in $\gra A$  with $a_n\in U_n,  n\in\NN$ such that
\begin{align}
\langle z,a^*_n\rangle+\langle a_n,z^*\rangle-\langle a_n, a^*_n\rangle
>\langle z,z^*\rangle,\quad\forall n\in\NN.\label{LSD:3}\end{align}
Then by \eqref{See:1a4} and \eqref{ree:2} we have
 \begin{align}a_n\neq 0,\quad\forall n\in\NN\quad\text{and}\quad
 a_n\longrightarrow z.\label{LSD:8}\end{align}

Now we claim that \begin{align}z\in\overline{\dom B}.\end{align}
Suppose to the contrary that $z\not\in\overline{\dom B}$.

As \eqref{LSD:8},  there exists $K_1\in\NN$ such that $a_n\notin\dom
B, \forall n\geq K_1$.  For convenience, we can and do suppose that
\begin{align}
a_n\notin\dom B,\quad\forall n\in\NN.\label{Lrtg:10}
\end{align}
Since $0\in\inte \dom B$,  by Fact~\ref{Ll:l1} and
Fact~\ref{f:referee02c},
 there exists $\delta\in\left]0,1\right[$
such that
\begin{align}\delta z \in\bd\overline{\dom B}.\label{LSD:4}\end{align}

Now  consider the operator: $B+N_{\left[0,a_n\right]}$. Following
the corresponding lines of the proof of Lemma~\ref{rcf:01} and by
\eqref{LSD:8}, $B+N_{\left[0,a_n\right]}$ is maximally monotone for
every $n\in\NN$.

Because $a_n\notin  \dom B$ by \eqref{Lrtg:10}, $a_n\not\in\dom
B\cap \left[0,a_n\right]=\dom (B+N_{\left[0,a_n\right]})$ for every
$n\in\NN$. Thus, $(a_n,z^*)\notin \gra (B+N_{\left[0,a_n\right]})$.
Thence
  there exist $\beta_n\in\left[0,1\right]$, $w^*_n\in B(\beta_n a_n)$ and $v^*_n\in
N_{\left[0,a_n\right]} (\beta_n a_n)$  such that
\begin{align}
\langle a_n-\beta_n a_n,z^*-w^*_n\rangle&<\langle a_n-\beta_n a_n,v^*_n\rangle\leq0,\quad\forall n\in\NN.
\label{LSD:7}
\end{align}

Since $\beta_n\in \left[0,1\right]$,  there is  a convergent
subsequence of $(\beta_n)_{n\in\NN}$, which, for convenience, we
still denote by $(\beta_n)_{n\in\NN}$. Now $\beta_n\longrightarrow
\beta$, where $\beta\in\left[0,1\right]$. Then by \eqref{LSD:8},
\begin{align}
\beta_n a_n\longrightarrow \beta z.\label{LSD:11}\end{align} We
claim that
\begin{align}
\beta\leq\delta<1.\label{LSD:12}\end{align}

Indeed, suppose to the contrary that $\beta>\delta$. By
\eqref{LSD:11}, $\beta z\in\overline{\dom B}$.
 Then by $0\in\inte\dom B$ and \cite[Theorem~1.1.2(ii)]{Zalinescu},
$\delta z=\tfrac{\delta}{\beta } \beta z\in \inte\overline{\dom B},$
which contradicts \eqref{LSD:4}. Hence \eqref{LSD:12} holds.

We can and do suppose that $\beta_n< 1$  for every $n\in\NN$. By
\eqref{LSD:7},
\begin{align}
\langle a_n,z^*-w^*_n\rangle<0,\quad\forall n\in\NN.\label{Lrtg:11}
\end{align}
Since $(0,0)\in\gra A$, $\langle a_n, a^*_n\rangle\geq 0, \forall n\in\NN$.
Then by \eqref{LSD:3}, we have
\begin{align}&\langle  z, \beta_n a^*_n\rangle+
\langle \beta_n a_n,z^*\rangle-\beta^2_n \langle a_n, a^*_n\rangle
 \geq\langle \beta_n z,a^*_n\rangle+\langle \beta_n a_n,z^*\rangle
 -\beta_n \langle a_n, a^*_n\rangle
\geq\beta_n \langle z,z^*\rangle.\label{LSD:17}\end{align} Hence,
by \eqref{LSD:17},
\begin{align}
\langle  z-\beta_n a_n, \beta_n a^*_n\rangle\geq\langle \beta_n z-\beta_n a_n, z^*\rangle.
\label{LSD:20}
\end{align}
Since $\gra A$ is a linear subspace and $(a_n,a^*_n)\in\gra A$, $(\beta_n a_n, \beta_n a^*_n)\in\gra A$.
 By \eqref{See:1a4}, we have
 \begin{align*}
 \lambda<&\langle z-\beta_n a_n, z^*-w^*_n-\beta_n a^*_n\rangle
 = \langle z-\beta_n a_n, z^*-w^*_n\rangle+\langle z-\beta_n a_n, -\beta_n a^*_n\rangle\\
 &\leq \langle z-\beta_n a_n, z^*-w^*_n\rangle
 -\langle \beta_n z-\beta_n a_n, z^*\rangle\quad\text{(by \eqref{LSD:20})}.
 \end{align*}
Then
\begin{align}
 \lambda
\leq
 \langle z-\beta_n a_n, z^*-w^*_n\rangle
 -\langle \beta_n z-\beta_n a_n, z^*\rangle.\label{LSD:21}
 \end{align}

We again consider two cases:

\emph{Case 1}:
 $(w^*_n)_{n\in\NN}$ is bounded.
 By the Banach-Alaoglu Theorem
(see \cite[Theorem~3.15]{Rudin}), there  exist a weak* convergent \emph{subnet}
$(w^*_\gamma)_{\gamma\in\Gamma}$ of $(w^*_n)_{n\in\NN}$ such that
\begin{align}{w^*_\gamma}\weakstarly  w^*_{\infty}\in X^*.\label{FCGG:9}\end{align}

Combine \eqref{LSD:8}, \eqref{LSD:11} and \eqref{FCGG:9}, we
pass to the limit along the given subnet of \eqref{LSD:21} to deduce that
\begin{align}
 \lambda\leq
 \langle z-\beta z, z^*-w^*_{\infty}\rangle.\label{LSD:22}
 \end{align}
By \eqref{LSD:12}, on dividing by $(1-\beta)$ on both sides of \eqref{LSD:22} we get
\begin{align}
\langle z,z^*-w^*_{\infty}\rangle\geq\frac{\lambda}{1-\beta}>0.\label{LSD:23}
\end{align}
On the other hand, by \eqref{LSD:8} and \eqref{FCGG:9}, on passing to the limit
along the  same subnet in \eqref{Lrtg:11} we see that
\begin{align}
\langle z, z^*-w^*_{\infty}\rangle\leq 0,\end{align}
which contradicts \eqref{LSD:23}.

\emph{Case 2}:
 $(w^*_n)_{n\in\NN}$ is unbounded.
After passing to a subsequence if necessary, we assume that
$\|w^*_n\|\neq 0,\forall n\in\NN$ and that $\|w^*_n\|\longrightarrow +\infty$.
By the Banach-Alaoglu Theorem
again, there  exist a weak* convergent \emph{subnet}
$(w^*_\nu)_{\nu\in I }$ of $(w^*_n)_{n\in\NN}$ such that
\begin{align}\frac{w^*_{\nu}}{\|w^*_{\nu}\|}\weakstarly  y^*_{\infty}\in X^*.\label{FCGG:LSD9}\end{align}

As $0\in\inte\dom B $ and using Fact~\ref{extlem},
there exist $\rho>0$ and $M>0$ such that
\begin{align}
\langle \beta_n a_n, \tfrac{w^*_n}{\|w^*_n\|}\rangle
\geq\rho-\frac{(\|\beta_n a_n\|+\rho) M}{\|w^*_n\|}, \quad \forall n\in\NN.\label{FCG:3}
\end{align}

Combining \eqref{LSD:11} and \eqref{FCGG:LSD9},and  taking the limit in \eqref{FCG:3} along the subnet, we obtain
\begin{align}
\langle \beta z, y^*_{\infty}\rangle
\geq\rho.\label{LSD:30}\end{align}
Then we have $\beta\neq0$ and thus $\beta>0$. By \eqref{LSD:30},
\begin{align}
\langle  z, y^*_{\infty}\rangle
\geq\frac{\rho}{\beta}>0.\label{LSD:34}\end{align}

Dividing by  $\|w^*_n\|$ in \eqref{LSD:21} and  taking the weak$^{*}$ limit in \eqref{LSD:21} along the subnet,
it follows from \eqref{LSD:11} and \eqref{FCGG:LSD9} that
\begin{align}
\langle z-\beta z, -y^*_{\infty}\rangle\geq0.
\label{LSD:32}
\end{align}

By \eqref{LSD:12},
\begin{align*}
 \langle z , y^*_{\infty}\rangle\leq0,\quad\text{which contradicts \eqref{LSD:34}}.
 \end{align*}

Combining all the cases above, we obtain $ z\in\overline{\dom B}$.

Next, we show that
\begin{align}
F_{A+B}(t z,tz^*)\geq t^2\langle z,z^*\rangle,
\quad\forall t\in\left]0,1\right[.\label{See:10}\end{align}
Let $t\in\left]0,1\right[$.
By  $0\in\inte\dom B$, Fact~\ref{f:referee02c} and \cite[Theorem~1.1.2(ii)]{Zalinescu}, we have
\begin{align}
tz\in\inte\overline{\dom B}\label{ReAu:1}.
\end{align}
 Fact~\ref{f:referee02c} implies that
\begin{align}tz\in\inte\dom B.
\end{align}

Set
 \begin{align*}
H_n:= t z+\tfrac{1}{n}U_X,\quad \forall n\in\NN.
\end{align*}
Since $\dom A$ is a linear subspace, $tz\in\overline{\dom A}\backslash {\dom A}$ by \eqref{LSD:2}.
Then $H_n\cap \dom A\neq\varnothing$.
Since $(tz, t z^*)\notin\gra A$ and $tz\in H_n$, and
$A$ is of type (FPV) by Fact~\ref{S:referee01},
 there exists $(b_n,b^*_n)_{n\in\NN}$ in $\gra A$
such that $b_n\in H_n$  and
\begin{align}
\langle t z,b^*_n\rangle+\langle b_n,t z^*\rangle-
\langle b_n,b^*_n\rangle>t^2\langle z,z^*\rangle,\quad \forall n\in\NN.\label{see:20}
\end{align}

As $t z\in\inte\dom B$ and $b_n\longrightarrow t z$, by Fact~\ref{extlem},
 there exist $N\in\NN$ and  $K>0$ such that
\begin{align}b_n\in\inte\dom B \quad\text{and}\quad \sup_{v^*\in B(b_n)}\| v^*\|
\leq K,\quad \forall n\geq N.\label{See:1a2}
\end{align}
Hence
\begin{align}
F_{A+B}(t z,t z^*)
&\geq\sup_{\{ c^*\in B(b_n)\}}\left[\langle b_n,t z^*\rangle
+\langle t z,b_n^*\rangle-\langle b_n,b_n^*\rangle
+\langle t z-b_n, c^*\rangle \right],\quad \forall n\geq N\nonumber\\
&\geq\sup_{\{ c^*\in B(b_n)\}}\left[t^2\langle z,z^*\rangle
+\langle t z-b_n,c^*\rangle \right],\quad \forall n\geq N
\quad\text{(by \eqref{see:20})}\nonumber\\
&\geq\sup\left[t^2\langle z,z^*\rangle
-K\|t z-b_n\| \right],\quad \forall n\geq N\quad\text{(by \eqref{See:1a2})}\nonumber\\
&\geq t^2\langle z,z^*\rangle\quad\text{(by $b_n\longrightarrow t z$)}.\label{See:1a3}
\end{align}
Hence $
F_{A+B}(t z,t z^*)\geq t^2\langle z,z^*\rangle$.

We have proved that \eqref{See:10} holds.
Since $(0,0)\in\gra (A+ B)$ and $A+B$ is monotone,
 we have $F_{A+B}(0,0)=\langle 0,0\rangle=0$.
Since $F_{A+B}$ is convex, \eqref{See:10} implies that
\begin{align*}
tF_{A+B}(z,z^*)=t F_{A+B}(z,z^*)+(1-t)F_{A+B}(0,0)\geq F_{A+B}(t z,tz^*)\geq t^2\langle z,z^*\rangle,
\quad\forall t\in\left]0,1\right[.
\end{align*}
Letting $t\longrightarrow 1^{-}$ in the above inequality, we obtain
\begin{align}
F_{A+B}(z,z^*)\geq \langle z,z^*\rangle.
\end{align}
Therefore,
 \eqref{e0:ourgoal} holds, and $A+B$ is maximally monotone.
\end{proof}

\begin{remark}
Theorem~\ref{PGV:1} generalizes the main results in \cite{BWY4,BWY9,Yao2}.

\end{remark}

We now establish the promised corollary:

\begin{corollary}[FPV property of the sum]\label{domain:L1}
Let $A:X\To X^*$ be  a maximally monotone linear relation. Let $B:X\rightrightarrows X^*$ be maximally monotone. If
 $\dom  A\cap\inte\dom B\neq\varnothing$,
then $A+B$ is of type  $(FPV)$.

\end{corollary}
\begin{proof} By Theorem~\ref{PGV:1}, $A+B$ is maximally monotone.
Let $C$ be a nonempty closed convex subset of $X$,
and suppose that $\dom (A+B) \cap \inte C\neq \varnothing$.
Let $x_1\in \dom A \cap \inte\dom B$ and $x_2\in \dom (A+B) \cap \inte C$.
Then $x_1,x_2\in\dom A$, $x_1\in\inte\dom B$ and
$x_2\in\dom B\cap \inte C$.
Hence $\lambda x_1+(1-\lambda)x_2\in\inte\dom B$ for every $\lambda\in
\left]0,1\right]$ by Fact~\ref{f:referee02c} and \cite[Theorem~1.1.2(ii)]{Zalinescu} and so there exists $\delta\in\left]0,1\right]$ such that
$\lambda x_1+(1-\lambda)x_2\in\inte C$ for every $\lambda\in\left[0,\delta\right]$.

Thus, $\delta x_1+(1-\delta)x_2\in\dom A\cap\inte\dom B\cap\inte C$.
By Fact~\ref{f:referee01} or \cite[Theorem~9(i)]{Bor2}, $B+N_C$ is maximally monotone.
Then, by Theorem~\ref{PGV:1} (applied  $A$ and $B+N_C$ to $A$ and $B$),
$A+B+N_C=A+(B+N_C)$ is maximally monotone.
By Fact~\ref{f:referee02a},   $A+B$ is of type  $(FPV)$.
\end{proof}

Note that with $A=0$ we recover the fact that a maximally monotone mapping is type (FPV) when its domain has nonempty interior.
\begin{remark}
The proof of Corollary~\ref{domain:L1} was adapted from that of \cite[Corollary~3.3]{Yao2}. Moreover,
Corollary~\ref{domain:L1} generalizes \cite[Corollary~3.3]{Yao2}.
\end{remark}

\section{Further Consequences}\label{sec:more}

The next result reduces all sum theorems to linear ones.

\begin{proposition}\label{ProC:V1}
Let $A, B: X\rightrightarrows X^*$ be monotone such that $\dom A\cap\dom B\neq\varnothing$,  let
\[C:=\{(x,x)\in X\times X\mid x\in X\}.\] Let $T: X \times X \rightrightarrows X^*\times X^*$ be defined
   by \[T(x,y):=\big(Ax,By\big).\]
   Then $A+B$ is maximally monotone if and only if $ T+N_C$ is maximally monotone.
\end{proposition}

\begin{proof}
We have
\begin{align}
N_C(x,x)=\big\{(x^*,-x^*)\mid x^*\in X^*\big\},\quad \forall x\in X.\label{ProCV1:e1}
\end{align}
``$\Rightarrow$":  Clearly, $T+N_C$ is monotone.   Let $\big((x_0, y_0), (x^*_0, y^*_0)\big)\in (X\times X)\times (X^*\times X^*)$ be monotonically related to $\gra (T+N_C)$.
Now we show that $\big((x_0, y_0), (x^*_0, y^*_0)\big)\in\gra (T+N_C)$.
Then by \eqref{ProCV1:e1},
\begin{align*}
&\Big\langle (x_0, y_0)-(a,a),(x^*_0, y^*_0)-(a^*, b^*)-(x^*,-x^*)\Big\rangle\geq0,\quad
\forall (a, a^*)\in\gra A, (a,b^*)\in\gra B,\\
&\quad \forall x^*\in X^*\\
&\Rightarrow\Big\langle (x_0-a, y_0-a),(x^*_0-a^*-x^*, y^*_0-b^*+x^*)\Big\rangle\geq0,\quad
\forall (a, a^*)\in\gra A, (a,b^*)\in\gra B,\\
&\quad \forall x^*\in X^*\\
&\Rightarrow \big\langle x_0-a, x^*_0-a^*\big\rangle+\big\langle y_0-a, y^*_0-b^*\big\rangle\geq0,\quad  \big\langle x_0-a, -x^*\big\rangle+\big\langle y_0-a, x^*\big\rangle=0,\\
&\quad
\forall (a, a^*)\in\gra A, (a,b^*)\in\gra B,\forall x^*\in X^*\\
&\Rightarrow \big\langle x_0-a, x^*_0-a^*\big\rangle+\big\langle y_0-a, y^*_0-b^*\big\rangle\geq0,\, x_0=y_0,\quad
\forall (a, a^*)\in\gra A, (a,b^*)\in\gra B\\
&\Rightarrow \big\langle x_0-a, x^*_0+y^*_0-a^*-b^*\big\rangle\geq0,\, x_0=y_0,\quad
\forall (a, a^*)\in\gra A, (a,b^*)\in\gra B\\
&\Rightarrow x^*_0+y^*_0\in(A+B)x_0,\, x_0=y_0\quad\text{(since $A+B$ is maximal monotone)}\\
&\Rightarrow \exists v^*\in X^*,\quad x^*_0+v^*\in Ax_0,  y^*_0-v^*\in Bx_0,\, x_0=y_0\\
&\Rightarrow \big((x_0, y_0), (x^*_0, y^*_0)\big)\in\gra (T+N_C)\quad\text{(by \eqref{ProCV1:e1})}.\\.
\end{align*}
Hence $T+N_C$ is maximally monotone.
\allowdisplaybreaks

``$\Leftarrow$":  Let $(z, z^*)\in X\times X^*$ be monotonically related to $\gra (A+B)$.
\begin{align*}
&\big\langle z-a, z^*-a^*-b^*\rangle\geq0,\quad\forall (a, a^*)\in\gra A, (a,b^*)\in\gra B\\
&\Rightarrow \big\langle z-a, \frac{z^*}{2}-a^*\big\rangle+\big\langle z-a, \frac{z^*}{2}-b^*\big\rangle+
\big\langle z-a, -x^*\big\rangle+\big\langle z-a, x^*\big\rangle\geq0,\\
&\quad
\forall (a, a^*)\in\gra A, (a,b^*)\in\gra B,\forall x^*\in X^*\\
&\Rightarrow \big\langle z-a, \frac{z^*}{2}-a^*-x^*\big\rangle+\big\langle z-a, \frac{z^*}{2}-b^*+x^*\big\rangle\geq0,\\
&\quad
\forall (a, a^*)\in\gra A, (a,b^*)\in\gra B,\forall x^*\in X^*\\
&\Rightarrow \Big\langle (z, z)-(a,a),(\frac{z^*}{2}, \frac{z^*}{2})-(a^*, b^*)-(x^*,-x^*)\Big\rangle\geq0,\quad
\forall (a, a^*)\in\gra A, (a,b^*)\in\gra B,\\
&\quad \forall x^*\in X^*\\
&\Rightarrow \Big\langle (z, z)-\mathbf{w},(\frac{z^*}{2}, \frac{z^*}{2})-\mathbf{w^*}\Big\rangle\geq0,\quad
\forall (\mathbf{w}, \mathbf{w^*})\in\gra (T+N_C)\quad\text{(by \eqref{ProCV1:e1})}\\
&\Rightarrow (\frac{z^*}{2}, \frac{z^*}{2})\in (T+N_C)(z,z)\quad\text{(since $T+N_C$ is maximal monotone)}\\
&\Rightarrow \exists v^*,\quad
 (\frac{z^*}{2}, \frac{z^*}{2})\in (Az, Bz)+(v^*,-v^*)\quad\text{(by \eqref{ProCV1:e1})}.\\
&\Rightarrow z^*=\frac{z^*}{2}+\frac{z^*}{2}\in (A+B)z.
\end{align*}
Hence $A+B$ is maximally monotone.
\end{proof}

It is important to note that exchanging the roles of $A$ and $B$ in Theorem \ref{PGV:1} leads to a much tougher linear sum problem.

\begin{remark}\label{rem42}
Let $A,B$ be maximally monotone such that $\dom A\cap\inte\dom B\neq\varnothing$ (i.e., they satisfy Rockafellar's constraint qualification), let $T, C$ be defined as in Proposition~\ref{ProC:V1}.
 Then we have \[\bigcup_{\lambda>0} \lambda\left[\dom T-\dom N_C\right]=X\times X.\]  Note that $T$ is maximally monotone (see the corresponding lines
 of proof of \cite[Proposition~3.13]{BY2}) and $N_C$ is a maximally monotone linear relation.  If the following conjecture is true then by Proposition~\ref{ProC:V1}, $A+B$ is maximally monotone
 and hence the general sum theorem holds.
\begin{quote}\emph{\textbf{Conjecture}
Let $S:X\To X^*$ be a maximally monotone linear relation,
and let $T: X\rightrightarrows X^*$ be maximally monotone
such that $\bigcup_{\lambda>0} \lambda\left[\dom S-\dom T\right]=X$.
Then $S+T$ is maximally monotone.}
\end{quote}
\end{remark}

In a related manner we have:

\begin{corollary}
Let $A:X\To X^*$ be a maximally monotone linear relation,
and let $B: X\rightrightarrows X^*$ be maximally monotone
such that $\dom A\cap\inte\dom B\neq\varnothing$.
Let
$C:=\{(x,x)\in X\times X\mid x\in X\}$ and let $T: X \times X \rightrightarrows X^*\times X^*$ be defined
   by $T(x,y):=\big(Ax,By\big)$.  Then
$T+N_C$ is maximally monotone.

\end{corollary}

\begin{proof}
Apply Theorem~\ref{PGV:1} and Proposition~\ref{ProC:V1} directly.
\end{proof}

In consequence, we are left with the following unresolved and interesting questions.
\begin{problem}
Let $A:X\rightarrow X^*$ be a continuous monotone linear operator,
and let $B:X\rightrightarrows X^*$ be maximally monotone.
  Is
$A+B$ necessarily maximally monotone $?$
\end{problem}

\begin{problem}
 Let $f:X\rightarrow \RX$ be a proper lower semicontinuous convex function,
and let $B:X\rightrightarrows X^*$ be maximally monotone with $\dom \partial f\cap\inte\dom B\neq\varnothing$.
  Is
$\partial f +B$ necessarily maximally monotone $?$
\end{problem}

Finally we recapitulate the conjecture after Remark \ref{rem42}.

\begin{problem}
Let $A:X\To X^*$ be a maximally monotone linear relation,
and let $B: X\rightrightarrows X^*$ be maximally monotone
such that $\bigcup_{\lambda>0} \lambda\left[\dom A-\dom B\right]=X$.
Is $A+B$ necessarily maximally monotone $?$

\end{problem}

\section*{Acknowledgment}
 The author thanks
Dr.\ Robert Csetnek for bringing us to consider various cases of the sum problem and also for his helpful comments.
Jonathan  Borwein and Liangjin Yao were partially supported
by various Australian Research Council grants.

\end{document}